\documentclass[12pt]{amsart}
\usepackage{amsmath, amssymb, amsthm, latexsym}
\usepackage{amssymb,amscd,amsmath}
\usepackage{latexsym}
\usepackage[center]{caption}
\usepackage{tikz}
\newcounter{braid}
\newcounter{strands}

\DeclareMathAlphabet{\bsf}{OT1}{cmss}{bx}{n}

\pgfkeyssetvalue{/tikz/braid height}{1cm}
\pgfkeyssetvalue{/tikz/braid width}{1cm}
\pgfkeyssetvalue{/tikz/braid start}{(0,0)}
\pgfkeyssetvalue{/tikz/braid colour}{black}
\pgfkeys{/tikz/strands/.code={\setcounter{strands}{#1}}}

\makeatletter
\def\cross{%
  \@ifnextchar^{\message{Got sup}\cross@sup}{\cross@sub}}

\def\cross@sup^#1_#2{\render@cross{#2}{#1}}

\def\cross@sub_#1{\@ifnextchar^{\cross@@sub{#1}}{\render@cross{#1}{1}}}

\def\cross@@sub#1^#2{\render@cross{#1}{#2}}

\def\render@cross#1#2{
  \def\strand{#1}
  \def\crossing{#2}
  \pgfmathsetmacro{\cross@y}{-\value{braid}*\braid@h}
  \pgfmathtruncatemacro{\nextstrand}{#1+1}
  \foreach \thread in {1,...,\value{strands}}
  {
    \pgfmathsetmacro{\strand@x}{\thread * \braid@w}
    \ifnum\thread=\strand
    \pgfmathsetmacro{\over@x}{\strand * \braid@w + .5*(1 - \crossing) * \braid@w}
    \pgfmathsetmacro{\under@x}{\strand * \braid@w + .5*(1 + \crossing) * \braid@w}
    \draw[braid] \pgfkeysvalueof{/tikz/braid start} +(\under@x pt,\cross@y pt) to[out=-90,in=90] +(\over@x pt,\cross@y pt -\braid@h);
    \draw[braid] \pgfkeysvalueof{/tikz/braid start} +(\over@x pt,\cross@y pt) to[out=-90,in=90] +(\under@x pt,\cross@y pt -\braid@h);
    \else
    \ifnum\thread=\nextstrand
    \else
     \draw[braid] \pgfkeysvalueof{/tikz/braid start} ++(\strand@x pt,\cross@y pt) -- ++(0,-\braid@h);
    \fi
   \fi
  }
  \stepcounter{braid}
}

\tikzset{braid/.style={double=\pgfkeysvalueof{/tikz/braid colour},double distance=1pt,line width=2pt,white}}

\newcommand{\braid}[2][]{%
  \begingroup
  \pgfkeys{/tikz/strands=2}
  \tikzset{#1}
  \pgfkeysgetvalue{/tikz/braid width}{\braid@w}
  \pgfkeysgetvalue{/tikz/braid height}{\braid@h}
  \setcounter{braid}{0}
  \let\sigma=\cross
  #2
  \endgroup
}
\makeatother

\input xypic
\newtheorem{theorem}{Theorem}

\newtheorem{lemma}[theorem]{Lemma}


\def\Z{\mathbb{Z}}

\def\C{\mathbb{C}}

\def\C{\mathbb{C}}

\def\N{\mathbb{N}}

\def\md{\mathcal{D}}

\def\qed{\hfill$\square$\medskip}

\def\Zpk{\mathbb{Z}/p^{k}}
\def\Zpk1{\mathbb{Z}/p^{k-1}}

\newcommand{\rref}[1]{(\ref{#1})}

\newcommand{\beg}[2]{\begin{equation}\label{#1}#2\end{equation}}
\def\r{\rightarrow}

\def\sl2{\widetilde{SL_{2}(\Z)}}

\def\md
\def\rank{\operatorname{rank}}

\title[Noether's problem for orientation $p$-subgroups]{Noether's problem for orientation $p$-subgroups of symmetric groups}
\author{Sophie Kriz}

\begin{document}
\maketitle

\begin{abstract}
We give a positive solution to Noether's rationality problem for certain index $p$ subgroups of the $p$-Sylow subgoups of symmetric groups.
\end{abstract}

\section{Introduction}

Let $G$ be a finite group acting linearly on an affine space over $\C$.
Emmy Noether \cite{noether} asked whether the quotient variety is always rational. A counterexample
was found by Saltman \cite{salt} (see also Bogomolov \cite{b87}), using an invariant called the unramified Brauer group. In a different form, 
this invariant was first discovered by Artin and Mumford \cite{am} who used it to give a different example of unirational varieties which are
not rational. (For a smooth projective variety over $\C$, the unramified Brauer group is isomorphic to the torsion in its third singular
cohomology group.) 
Colliot-Th\'el\`ene and Ojanguren \cite{cothe} and Peyre \cite{peyre}
found more examples of unirational non-rational varieties using a more general invariant
called unramified cohomology.
The question for which groups $G$ Noether's problem has a positive solution is still widely open except
in some special cases. For abelian groups, a positive solution was given by Fischer \cite{fischer}. For 
groups of order $p^n$, $n\leq 4$, it was solved positively by Chu and Kang \cite{ckang}. 
It is proved in \cite{KWZ}, Theorem 1.4, that for $p$-Sylow subgroups of
symmetric groups, Noether's problem also has a positive solution (see \cite{bobo11,bobo12}), and some other
cases are also known \cite{maeda,ckang30,ckang10,ckang20,ckang1}.
The most famous
case of Noether's problem is $G=A_n$ (the rationality of the discriminant variety).  It is known to be true for $n\leq 5$
(the most interesting case is Maeda's Theorem \cite{maeda} for $n=5$), but is still open for $n\geq 6$.
It was proved by Bogomolov and Petrov \cite{bop11} that unramified cohomology is $0$ in this case.

\vspace{3mm}
In this note, we consider Noether's problem for a class of index $p$ subgroups of $p$-Sylow subgroups of
symmetric groups. For our purposes, it is convenient to choose certain particular generators of those subgroups. 
A $p$-Sylow subgroup of the symmetric group $\Sigma_n$ on $\{1,\dots,n\}$ has generators $\sigma_{i,s}$
which cyclically permute $p$ consecutive blocks of $p^i$ consecutive numbers $1,\dots, n$, starting with a number
congruent to $1\mod p^{i+1}$. The group $H_{n,p}$ is the index $p$ subgroup generated by $\sigma_{0,s}\sigma_{0,
s^\prime}^{-1}$ and $\sigma_{i,s}$ for $i>0$. For more detail on these groups, see Section \ref{snot} below.

The group $H_{8,3}$ is
the $3$-Sylow subgroup of the group of positions of the corners of 
Rubik's cube. Because of this, we call $H_{n,p}$ the {\em orientation $p$-subgroup of 
$\Sigma_n$}. 

\vspace{3mm}
The following is our main theorem.

\begin{theorem}\label{ts}
For a field $F$ of characteristic $0$ containing $p$'th roots of unity, let $H_{n,p}$ act on 
$$F(x_1,\dots ,x_n)$$
by permuting the $x_i$'s. Then for every prime $p\geq 2$,
the field of fixed points $F(x_1,\dots ,x_n)^{H_{n,p}}$ is rational over $F$.
\end{theorem}

Note that $H_{n,2}$ is a $2$-Sylow subgroup of the alternating group $A_n$. Therefore, one obtains
as a corollary a simple proof of the result of Bogomolov and Petrov \cite{bop11}:

\begin{theorem} \label{tbop}
(\cite{bop11}) We have
$H_{nr}^k(\C(a_1,\dots,a_n)^{A_n},\Z/\ell)=0$.
\end{theorem}

For $p=2$, there is an elementary proof of Theorem \ref{ts} based on the fact by letting 
the commuting elements $\sigma_{1,s}$
act
on the (twisted) torus corresponding to the lattice $D_m^*$ (dual to the root lattice of type $D_m$ algebraic groups), the 
quotient variety is birationally equivalent 
to the product of an affine space with another such (twisted)
torus equivariantly with respect to the permutations $\sigma_{i,s}$
for $i>1$. This, in turn, is basically due to the simple fact that $\Sigma_2=\Z/2$. (See Comment below the 
proof of Lemma \ref{llllemma} in Section \ref{p2} below.) Since this proof is short and easy,
we present it first in Section \ref{p2} below.

\vspace{3mm}

For $p>2$, the proof of Theorem \ref{ts} is given in Section \ref{sp}. It is based on the following result (used in the proof of Lemma \ref{lwqm} in Section \ref{sp} below):

\begin{lemma}\label{tw}
If $F$ contains $p$'th roots of unity, and on $F(x_1,\dots,x_p)$, $\Z/p$ acts by
\beg{1tw}{
x_1\mapsto x_2 \mapsto\dots x_{p}\mapsto x_1,}
then $F(x_1,\dots,x_p)^{\Z/p}$ is rational over $F(w)$ where $w$ denotes the product
$x_1\cdot\dots\cdot x_p$.
\end{lemma}

This Lemma is a corollary of the following result of Hajja:

\begin{lemma}(\cite{Hajjamon}, Lemma 2 (iv), p.244)\label{thajja}
Let $p$ be an odd prime number, $k$ a field of characteristic not equal to $p$. Assume that $k$ contains a
primitive $p$'th root of unity. If $a\in k^\times$, and $\sigma$ acts on the rational function field $k(x_1,\dots,x_{p-1})$
by
$$\sigma:x_1\mapsto x_2\mapsto\dots\mapsto x_{p-1}\mapsto a/(x_1\dots x_{p-1}),$$
then the fixed field $k(x_1,\dots,x_{p-1})^{\langle\sigma\rangle}$ is rational over $k$.
\end{lemma}

This implies Lemma \ref{tw} by setting $k=F(w)$ (considered as a subfield of $F(x_1,\dots,x_p)$), 
$a=w$.

\vspace{3mm}

For example for $p=3$, consider the cubic equation
$$x^3+bx^2+cx+w=0.$$
If we denote by $D$ the discriminant
$$D=18bcw-4b^3w+b^2c^2-4c^3-27w^2,$$
Lemma \ref{tw} says that $F(b,c,w)(\sqrt{D})$ is rational over $F(w)$. 
To make the proof of Theorem \ref{ts} self-contained, we give a proof of Lemma \ref{tw} in
Section \ref{sp} below.

\vspace{3mm}

We also note that there is an easy analogue of Theorem \ref{ts} for a larger group. Consider the homomorphism
$$\phi:\Sigma_{n}\wr \Z/p\r \Z/p$$
given on the normal subgroup $(\Z/p)^n$ by adding the coordinates of each element, 
and trivial on the subgroup $\Sigma_n$ which acts on $(\Z/p)^n$ by permutation of factors. Explicitly,
representing an element of $\Sigma_n\wr\Z/p$ as a tuple $(a_1,\dots, a_n;\tau)$ where $a_i\in \Z/p$, 
$\tau\in \Sigma_n$, we put
$$\phi(a_1,\dots,a_n;\tau):=\sum_{i=1}^{n} a_i.$$
Let $K_{n,p}=Ker(\phi)$.
Then $H_{np,p}$ is a $p$-Sylow subgroup of $K_{n,p}$. 

\begin{theorem}\label{tpos}
For all $n$ and $p$, the field $\C(x_1,\dots,x_{np})^{K_{n,p}}$ where $K_{n,p}$ acts by permutation
of variables is rational.
\end{theorem}

\vspace{3mm}
\noindent 
{\bf Acknowledgement:} I am thankful to Professor M.C.Kang for helpful advice on this paper.

\vspace{3mm}

\section{Conventions and notation} \label{snot}

The main purpose of this section is to explain in detail the construction of the groups $G_{n,p}$, $H_{n,p}$, and 
our notation for the generators.

\vspace{3mm}
\begin{lemma} 
\label{lsnot}
A $p$-Sylow subgroup $G_{n,p}\subset\Sigma_n$ (the group of all permutations on $\{1,\dots,n\}$) is generated by permutations $\sigma_{i,s}$ consisting of the cycles
$$(s-1)p^{i+1}+j\mapsto (s-1)p^{i+1}+ p^i+j\mapsto\dots\mapsto (s-1)p^{i+1}+ (p-1)p^i+j$$
$$\mapsto (s-1)p^{i+1}+j$$
for $1\leq j\leq p^i$. (The remaining elements are fixed.) The parameters $i,s$ range over integers satisfying
\beg{ssi+}{1\leq s\leq \lfloor\frac{n}{p^{i+1}}\rfloor, \; 0\leq i\leq \lfloor \log_p(n)\rfloor -1.}
\end{lemma}
\begin{proof}
It is well known that for $n=\ell_0+\ell_1 p+\dots+\ell_m p^m$, where $0\leq\ell_i\leq p-1$ are integers and we have $m=\lfloor \log_p(n) \rfloor$
(this implies $\ell_m >0$), a $p$-Sylow subgroup of $\Sigma_n$ is isomorphic to
$$\prod_{i=0}^{m-1}  (\underbrace{\Z/p\wr \dots\wr \Z/p}_{i+1})^{\ell_{i+1}}.$$
Now, by definition, the generators 
\beg{gen}{\sigma_{m-k,1},\dots, \sigma_{m-k,p^{k-1}},\; k=1,\dots,m}
generate a group isomorphic to  $G_{p^m,p}$ and commute with all the other generators. We will show that
\beg{ewrwr}{G_{p^m,p}\cong\underbrace{\Z/p\wr \dots\wr \Z/p}_m.}
Therefore, we obtain an isomorphism 
$$G_{n,p}\cong (\underbrace{\Z/p\wr \dots\wr \Z/p}_m) \times G_{n-p^m,p},$$
by sending the generators \rref{gen} to the generators of the first factor, and the remaining generators by
$$\sigma_{m-k,s}\mapsto \sigma_{m-k,s-p^{k-1}},$$
for $s>p^{k-1}.$ Then our statement follows by induction on $n$.

To show \rref{ewrwr}, recall that for a permutation group $G$ on $\ell$ elements
and any group $H$, the wreath product $G\wr H$
is defined as the semidirect product $G\ltimes (H^\ell)$ where the action of $G$ on $H^\ell $ is by permutation
of factors. We can prove \rref{ewrwr} by induction on $m$. The statement is obvious for $m=1$. 
Now by definition, the generators
\beg{eggens}{\sigma_{m-k,1},\dots, \sigma_{m-k,p^{k-1}},\; k=2,\dots,m}
generate a group isomorphic to 
\beg{eggg}{(G_{p^{m-1},p})^{p}.}
Now 
for an element $(\alpha_1,\dots,\alpha_p)\in (G_{p^{m-1},p})^p$ presented in the generators \rref{eggens},
we have
$$\sigma_{m-1,1}\cdot(\alpha_1,\dots,\alpha_p)\cdot \sigma_{m-1,1}^{-1}=(\alpha_p,\alpha_1\dots,\alpha_{p-1}).$$
Additionally, for $i,j\in \{0,\dots,p-1\}$, $(\alpha_1,\dots,\alpha_p),(\beta_1,\dots,\beta_p)\in G_{p^{m-1},p}$,
we clearly have
$$\sigma_{m-1,1}^i\cdot(\alpha_1,\dots,\alpha_p)=\sigma_{m-1,1}^j\cdot(\beta_1,\dots,\beta_p)$$
if and only if $i=j$ and $(\alpha_1,\dots,\alpha_p)=(\beta_1,\dots,\beta_p)$, thus gving an isomorphism
$$G_{p^m,p}\cong \Z/p\wr G_{p^{m-1},p}.$$
If \rref{ewrwr} is true with 
$m$ replaced by $m-1$, then by the induction hypothesis, 
$G_{p^{m-1},p}$ is isomorphic to
$$(\underbrace{\Z/p\wr \dots\wr \Z/p}_{m-1}),$$
thus giving our statement.
\end{proof}

\vspace{3mm}
\noindent\textbf{Example:} Let $n = 1+2^2 + 2^3 =13$, $p =2$. Then 
\[ 0 \leq i \leq 2 = \lfloor \log_2(13) \rfloor -1  . \]
For $i =0$, we have $1 \leq s \leq 6 = \lfloor \frac{13}{2} \rfloor$. For $i = 1$, we have 
$1 \leq s \leq 3 = \lfloor \frac{13}{4} \rfloor$, and for $i =2$, we have $1 \leq s \leq 1 = \lfloor \frac{13}{8} \rfloor$. 
Hence, there are $6 +3 +1 = 10$ generators of the form $\sigma_{i, s}$. For example, for $\sigma_{0, 1}$, there is only 
$j = 1$, and the only cycle is 
\[ 1 \mapsto 2 \mapsto 1 . \]

The generators of $G_{13, 2}$ can be visualised as follows. 
\vspace{4mm}

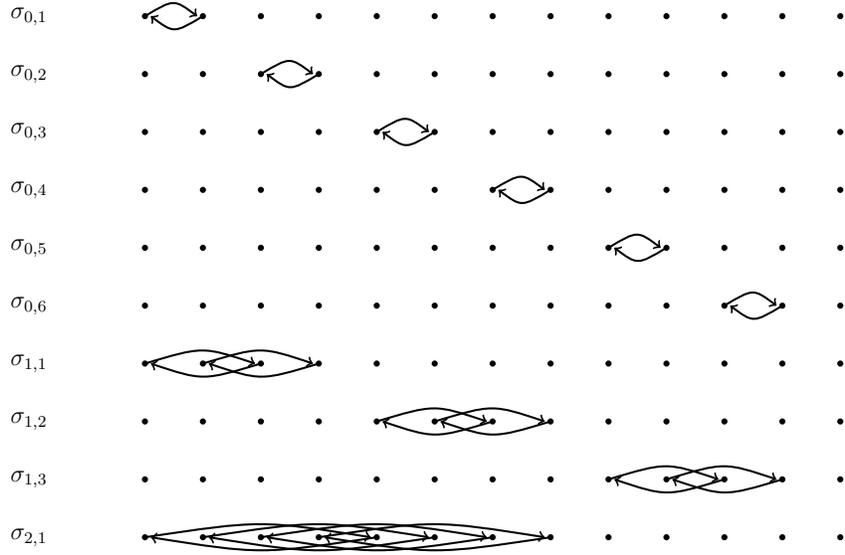
\resizebox{0.9\hsize}{!}{
\begin{tikzpicture}[line width=1pt][scale=1]
\draw [fill] (1, 1) circle [radius=1pt];
\draw [fill] (2, 1) circle [radius=1pt];
\draw [fill] (3, 1) circle [radius=1pt];
\draw [fill] (4, 1) circle [radius=1pt];
\draw [fill] (5, 1) circle [radius=1pt];
\draw [fill] (6, 1) circle [radius=1pt];
\draw [fill] (7, 1) circle [radius=1pt];
\draw [fill] (8, 1) circle [radius=1pt];
\draw [fill] (9, 1) circle [radius=1pt];
\draw [fill] (10, 1) circle [radius=1pt];
\draw [fill] (11, 1) circle [radius=1pt];
\draw [fill] (12, 1) circle [radius=1pt];
\draw [fill] (13, 1) circle [radius=1pt];
\draw [->] (1, 1) .. controls (1.5, 1.3) .. (1.9, 1);
\draw [<-] (1.1, 1) .. controls (1.5, 0.7) ..(2, 1);
\node  at (-1,1) {$\sigma_{0, 1}$};

\draw [fill] (1, 0) circle [radius=1pt];
\draw [fill] (2, 0) circle [radius=1pt];
\draw [fill] (3, 0) circle [radius=1pt];
\draw [fill] (4, 0) circle [radius=1pt];
\draw [fill] (5, 0) circle [radius=1pt];
\draw [fill] (6, 0) circle [radius=1pt];
\draw [fill] (7, 0) circle [radius=1pt];
\draw [fill] (8, 0) circle [radius=1pt];
\draw [fill] (9, 0) circle [radius=1pt];
\draw [fill] (10, 0) circle [radius=1pt];
\draw [fill] (11, 0) circle [radius=1pt];
\draw [fill] (12, 0) circle [radius=1pt];
\draw [fill] (13, 0) circle [radius=1pt];
\draw [->] (3, 0) .. controls (3.5, 0.3) .. (3.9, 0);
\draw [<-] (3.1, 0) .. controls (3.5, -0.3) ..(4, 0);
\node  at (-1,0) {$\sigma_{0, 2}$};

\draw [fill] (1, -1) circle [radius=1pt];
\draw [fill] (2, -1) circle [radius=1pt];
\draw [fill] (3, -1) circle [radius=1pt];
\draw [fill] (4, -1) circle [radius=1pt];
\draw [fill] (5, -1) circle [radius=1pt];
\draw [fill] (6, -1) circle [radius=1pt];
\draw [fill] (7, -1) circle [radius=1pt];
\draw [fill] (8, -1) circle [radius=1pt];
\draw [fill] (9, -1) circle [radius=1pt];
\draw [fill] (10, -1) circle [radius=1pt];
\draw [fill] (11, -1) circle [radius=1pt];
\draw [fill] (12, -1) circle [radius=1pt];
\draw [fill] (13, -1) circle [radius=1pt];
\draw [->] (5, -1) .. controls (5.5, -0.7) .. (5.9, -1);
\draw [<-] (5.1, -1) .. controls (5.5, -1.3) ..(6, -1);
\node  at (-1,-1) {$\sigma_{0, 3}$};

\draw [fill] (1, -2) circle [radius=1pt];
\draw [fill] (2, -2) circle [radius=1pt];
\draw [fill] (3, -2) circle [radius=1pt];
\draw [fill] (4, -2) circle [radius=1pt];
\draw [fill] (5, -2) circle [radius=1pt];
\draw [fill] (6, -2) circle [radius=1pt];
\draw [fill] (7, -2) circle [radius=1pt];
\draw [fill] (8, -2) circle [radius=1pt];
\draw [fill] (9, -2) circle [radius=1pt];
\draw [fill] (10, -2) circle [radius=1pt];
\draw [fill] (11, -2) circle [radius=1pt];
\draw [fill] (12, -2) circle [radius=1pt];
\draw [fill] (13, -2) circle [radius=1pt];
\draw [->] (7, -2) .. controls (7.5, -1.7) .. (7.9, -2);
\draw [<-] (7.1, -2) .. controls (7.5, -2.3) ..(8, -2);
\node  at (-1,-2) {$\sigma_{0, 4}$};

\draw [fill] (1, -3) circle [radius=1pt];
\draw [fill] (2, -3) circle [radius=1pt];
\draw [fill] (3, -3) circle [radius=1pt];
\draw [fill] (4, -3) circle [radius=1pt];
\draw [fill] (5, -3) circle [radius=1pt];
\draw [fill] (6, -3) circle [radius=1pt];
\draw [fill] (7, -3) circle [radius=1pt];
\draw [fill] (8, -3) circle [radius=1pt];
\draw [fill] (9, -3) circle [radius=1pt];
\draw [fill] (10, -3) circle [radius=1pt];
\draw [fill] (11, -3) circle [radius=1pt];
\draw [fill] (12, -3) circle [radius=1pt];
\draw [fill] (13, -3) circle [radius=1pt];
\draw [->] (9, -3) .. controls (9.5, -2.7) .. (9.9, -3);
\draw [<-] (9.1, -3) .. controls (9.5, -3.3) ..(10, -3);
\node  at (-1,-3) {$\sigma_{0, 5}$};

\draw [fill] (1, -4) circle [radius=1pt];
\draw [fill] (2, -4) circle [radius=1pt];
\draw [fill] (3, -4) circle [radius=1pt];
\draw [fill] (4, -4) circle [radius=1pt];
\draw [fill] (5, -4) circle [radius=1pt];
\draw [fill] (6, -4) circle [radius=1pt];
\draw [fill] (7, -4) circle [radius=1pt];
\draw [fill] (8, -4) circle [radius=1pt];
\draw [fill] (9, -4) circle [radius=1pt];
\draw [fill] (10, -4) circle [radius=1pt];
\draw [fill] (11, -4) circle [radius=1pt];
\draw [fill] (12, -4) circle [radius=1pt];
\draw [fill] (13, -4) circle [radius=1pt];
\draw [->] (11, -4) .. controls (11.5, -3.7) .. (11.9, -4);
\draw [<-] (11.1, -4) .. controls (11.5, -4.3) ..(12, -4);
\node  at (-1,-4) {$\sigma_{0, 6}$};

\draw [fill] (1, -5) circle [radius=1pt];
\draw [fill] (2, -5) circle [radius=1pt];
\draw [fill] (3, -5) circle [radius=1pt];
\draw [fill] (4, -5) circle [radius=1pt];
\draw [fill] (5, -5) circle [radius=1pt];
\draw [fill] (6, -5) circle [radius=1pt];
\draw [fill] (7, -5) circle [radius=1pt];
\draw [fill] (8, -5) circle [radius=1pt];
\draw [fill] (9, -5) circle [radius=1pt];
\draw [fill] (10, -5) circle [radius=1pt];
\draw [fill] (11, -5) circle [radius=1pt];
\draw [fill] (12, -5) circle [radius=1pt];
\draw [fill] (13, -5) circle [radius=1pt];
\draw [->] (1, -5) .. controls (2, -4.7) .. (2.9, -5);
\draw [<-] (1.1, -5) .. controls (2, -5.3) ..(3, -5);
\draw [->] (2, -5) .. controls (3, -4.7) .. (3.9, -5);
\draw [<-] (2.1, -5) .. controls (3, -5.3) ..(4, -5);
\node  at (-1,-5) {$\sigma_{1, 1}$};

\draw [fill] (1, -6) circle [radius=1pt];
\draw [fill] (2, -6) circle [radius=1pt];
\draw [fill] (3, -6) circle [radius=1pt];
\draw [fill] (4, -6) circle [radius=1pt];
\draw [fill] (5, -6) circle [radius=1pt];
\draw [fill] (6, -6) circle [radius=1pt];
\draw [fill] (7, -6) circle [radius=1pt];
\draw [fill] (8, -6) circle [radius=1pt];
\draw [fill] (9, -6) circle [radius=1pt];
\draw [fill] (10, -6) circle [radius=1pt];
\draw [fill] (11, -6) circle [radius=1pt];
\draw [fill] (12, -6) circle [radius=1pt];
\draw [fill] (13, -6) circle [radius=1pt];
\draw [->] (5, -6) .. controls (6, -5.7) .. (6.9, -6);
\draw [<-] (5.1, -6) .. controls (6, -6.3) ..(7, -6);
\draw [->] (6, -6) .. controls (7, -5.7) .. (7.9, -6);
\draw [<-] (6.1, -6) .. controls (7, -6.3) ..(8, -6);
\node  at (-1,-6) {$\sigma_{1, 2}$};

\draw [fill] (1, -7) circle [radius=1pt];
\draw [fill] (2, -7) circle [radius=1pt];
\draw [fill] (3, -7) circle [radius=1pt];
\draw [fill] (4, -7) circle [radius=1pt];
\draw [fill] (5, -7) circle [radius=1pt];
\draw [fill] (6, -7) circle [radius=1pt];
\draw [fill] (7, -7) circle [radius=1pt];
\draw [fill] (8, -7) circle [radius=1pt];
\draw [fill] (9, -7) circle [radius=1pt];
\draw [fill] (10, -7) circle [radius=1pt];
\draw [fill] (11, -7) circle [radius=1pt];
\draw [fill] (12, -7) circle [radius=1pt];
\draw [fill] (13, -7) circle [radius=1pt];
\draw [->] (9, -7) .. controls (10, -6.7) .. (10.9, -7);
\draw [<-] (9.1, -7) .. controls (10, -7.3) ..(11, -7);
\draw [->] (10, -7) .. controls (11, -6.7) .. (11.9, -7);
\draw [<-] (10.1, -7) .. controls (11, -7.3) ..(12, -7);
\node  at (-1,-7) {$\sigma_{1, 3}$};

\draw [fill] (1, -8) circle [radius=1pt];
\draw [fill] (2, -8) circle [radius=1pt];
\draw [fill] (3, -8) circle [radius=1pt];
\draw [fill] (4, -8) circle [radius=1pt];
\draw [fill] (5, -8) circle [radius=1pt];
\draw [fill] (6, -8) circle [radius=1pt];
\draw [fill] (7, -8) circle [radius=1pt];
\draw [fill] (8, -8) circle [radius=1pt];
\draw [fill] (9, -8) circle [radius=1pt];
\draw [fill] (10, -8) circle [radius=1pt];
\draw [fill] (11, -8) circle [radius=1pt];
\draw [fill] (12, -8) circle [radius=1pt];
\draw [fill] (13, -8) circle [radius=1pt];
\draw [->] (1, -8) .. controls (3, -7.7) .. (4.9, -8);
\draw [<-] (1.1, -8) .. controls (3, -8.3) ..(5, -8);
\draw [->] (2, -8) .. controls (4, -7.7) .. (5.9, -8);
\draw [<-] (2.1, -8) .. controls (4, -8.3) ..(6, -8);
\draw [->] (3, -8) .. controls (5, -7.7) .. (6.9, -8);
\draw [<-] (3.1, -8) .. controls (5, -8.3) ..(7, -8);
\draw [->] (4, -8) .. controls (6, -7.7) .. (7.9, -8);
\draw [<-] (4.1, -8) .. controls (6, -8.3) ..(8, -8);
\node  at (-1,-8) {$\sigma_{2, 1}$};
\end{tikzpicture}
}

\vspace{10mm}
Here $\sigma_{0,1}, \sigma_{0,2},\sigma_{0,3},\sigma_{0,4}, \sigma_{1,1},\sigma_{1,2},\sigma_{2,1}$ generate a copy of $\Z/2\wr\Z/2\wr\Z/2$ and $\sigma_{0,5},\sigma_{0,6},\sigma_{1,3}$ generate a copy of $\Z/2\wr\Z/2$.

\vspace{3mm}


Now consider the normal subgroup $H_{n,p}\subset G_{n,p}$ generated by all the
elements $\sigma_{0,s}\sigma_{0,s^\prime}^{-1}$ where $1\leq s, s'\leq\lfloor\frac{n}{p}\rfloor$, and 
$\sigma_{i,s}$ where \rref{ssi+} holds and $i\geq 1$. 

\begin{lemma}\label{lssnot}
Letting $G_{\lfloor n/p\rfloor,p}$ act on $(\Z/p)^{\lfloor \frac{n}{p}\rfloor}$ by permutation of factors, the group
$H_{n,p}$ is isomorphic to a semidirect product of $G_{\lfloor \frac{n}{p}\rfloor,p}$ with
the subgroup of $(\Z/p)^{\lfloor \frac{n}{p}\rfloor}$ of elements whose coordinates add up to $0$. Additionally,
$H_{n,p}$ is an index $p$ subgroup of $G_{n,p}$.
\end{lemma}

\begin{proof}
The subgroup $A$ of $G_{n,p}$ generated by $\sigma_{0,i}$ is a normal subgroup isomorphic to $(\Z/p)^{\lfloor n/p
\rfloor}$. Further, the quotient is isomorphic to $G_{\lfloor n/p\rfloor,p}$ (acting, instead of on single elements, on
$p$-tuples of consecutive elements of $\{1,\dots,p\lfloor n/p\rfloor\}$, each starting with a number $1\mod p$). 
Thus, the short exact sequence
$$1\r A\;\triangleleft\; G_{n,p}\r G_{\lfloor n/p\rfloor,p}\r 1$$
splits, and hence is a semidirect product. Now define a homomorphism $h:A=(\Z/p)^{\lfloor n/p\rfloor}\r\Z/p$
by adding the coordinates (in the category of abelian groups, 
this is the ``codiagonal"). If we denote, for $\alpha\in G_{\lfloor n/p\rfloor,p}$ and $q\in A$,
by $q^\alpha$ the image of $q$ under the automorphism of $A$ given by $\alpha$, then we have
$$h(q^\alpha)=h(q).$$
Therefore, the homomorphism $h$ extends uniquely to a homomorphism $\widetilde{h}:G_{n,p}\r \Z/p$
which is trivial on the subgroup $G_{\lfloor n/p\rfloor,p}$. The group $H_{n,p}$ is, by definition, the kernel of this homomorphism,
which gives us a diagram
$$\diagram
&1\dto &1\dto&&\\
1\rto & Ker(h)\dto\rto &Ker(\widetilde{h})\rto\dto & G_{\lfloor n/p\rfloor,p}\dto^\cong\rto &1\\
1\rto & A\dto^h\rto & G_{n,p}\dto^{\widetilde{h}}\rto & G_{\lfloor n/p\rfloor,p}\rto & 1\\
 &\Z/p\dto\rto^\cong &\Z/p \dto &&\\
& 1 & 1 &&
\enddiagram
$$
The columns are exact and the
middle row is exact, hence so is the top row. Further, the middle row splits, and the composition of $\widetilde{h}$
with the splitting is $0$. Thus, the splitting lifts to the top row.
\end{proof}

\vspace{3mm}
\noindent
{\bf Comment:} It follows from our proof of Lemma \ref{lsnot} that an element of $G_{n,p}$ can be 
uniquely written as a product (in any fixed order) of powers $\sigma_{i,s}^{r(i,s)}$ with $
r(i,s)\in \{0,\dots,p-1\}$, where $i,s$ are as in \rref{ssi+}.  Therefore, the group order of $G_{n,p}$ is $p$ to the
power equal to the number of pairs of integers $(i,s)$ satisfying \rref{ssi+}, which is easily checked (by an induction
mimicking that in the proof of Lemma \ref{lsnot}) to be the maximum power of $p$ dividing $n!$. Similarly, it follows from
our proof of Lemma \ref{lssnot} that an element of the group $H_{n,p}$ can be 
written uniquely as a product (in any fixed order) 
of powers $\sigma_{i,s}^{r(i,s)}$ with $
r(i,s)\in \{0,\dots,p-1\}$, where $i,s$ are as in \rref{ssi+} such that 
$$\sum_{s=1}^{\lfloor n/p\rfloor} r(0,s)\equiv 0\mod p.$$
This also shows that $|H_{n,p}|=|G_{n,p}|/p$.

\vspace{5mm}

\section{The case $p=2$.}\label{p2}
In this section, we prove Theorem \ref{ts} for $p=2$, and Theorem \ref{tbop}.

\vspace{3mm}
\noindent
{\em Proof of Theorem \ref{ts} for $p=2$:}
When $n$ is odd, the variable $x_n$ is fixed by $H_{n,2}$, so we can replace $F$ by $F(x_n)$ and $n$ by $n-1$.
Thus, we may assume $n=2m$ is even. We have $F(x_1,\dots ,x_n)=F(z_1,\dots ,z_m,t_1,\dots ,t_m)$
$$z_i=x_{2i-1}+x_{2i}$$
$$t_i=x_{2i-1}-x_{2i}$$
$$1\le i\le m.$$
Let $J_n\triangleleft H_{n,2}$ be the subgroup generated by the pairs of elements $\sigma_{0,s}\sigma_{0,s^\prime}$,
$1\leq s,s^\prime\leq m$.
Then 
$$F(x_1,\dots,x_n)^{J_n}=F(z_1,\dots, z_m, t_1^2,\dots, t_{m-1}^2,t_1t_2\cdots t_m).$$
The factor group $H_{n,2}/J_n$ is generated by
$\sigma_{i,s}$ where $i\geq 1$ and formula \rref{ssi+} holds.  Hence, it is isomorphic to the $2$-Sylow 
subgroup $G_{m,2}\subseteq \Sigma_m$ by the isomorphism $\sigma_{i,s}\mapsto\sigma_{i-1,s}
\in G_{m,2}$. The group $G_{m,2}$ acts on the $z_i$'s and $t_i$'s
by permutation. Therefore we have reduced the proof to the following statement. \qed

\begin{lemma}\label{llllemma}
Let $F$ be a field of characteristic $0$ and let $a^2\in F$ (in other words, $a$ is a square root of an element of $F$).
Let $E$ be the subfield of
\beg{efffield}{F[a](t_1,\dots,t_k,z_{1,1},\dots,z_{k,1},\dots,z_{1,\ell},\dots,z_{k,\ell})}
generated by the 
(algebraically independent) elements
$$\begin{array}{l}
t_1\cdots t_k a,t_1^2,t_2^2,\dots, t_{k-1}^2,\\
z_{1,1},z_{2,1},\dots,z_{k,1},\\
\vdots\\
z_{1,\ell},z_{2,\ell},\dots,z_{k,\ell}\\
\end{array}$$
and let $G_{k,2}$ act on $E$ by restriction of an action of $\Sigma_k$ on 
\rref{efffield} where $\Sigma_k$ acts trivially on $F[a]$, and acts by permuting the 
indexes $i$ in the $t_i$'s and the $z_{i,j}$'s for a fixed $j$.
Then $$E^{G_{k,2}}$$
is rational over $F[a]$.
\end{lemma}

\vspace{3mm}
\noindent
{\bf Comment:} The additional variables $z_{i,j}$, and the element $a$, are introduced in order to enable a proof
by induction.

\vspace{3mm}
\noindent
\begin{proof}
Induction on $k$.
If $k=1$, $E^{G_{k,2}}=E$.
If $k>1$ is odd, $G_{k,2}=G_{k-1,2}$.
Let $F'=F(t_k^2,z_{k,1},\dots,z_{k,\ell})$, $a'=t_ka$, $t_i'=t_i$ for $i<k$, and $z_{i,j}$ for $ i<k$,
$j=1,\dots, \ell$. (Note that $F^\prime[a^\prime]$ is rational over $F[a]$.)
This reduces the statement for $k$ to the statement for $k-1$. Thus, we may assume $k=2m$ is even.
 Let $\Gamma_k\triangleleft G_{k,2}$ be the subgroup generated by the elements, $\sigma_{0,s}$,
$1\leq s\leq k$.
Then 
$$G_{k,2}/\Gamma_k\cong G_{m,2}.$$
We have 
$$E^{G_{k,2}}=(E^{\Gamma_{k}})^{G_{k,2}/\Gamma_k}.$$
Further,
\beg{efff}{\begin{array}{c}E^{\Gamma_k}=F[a](u_1u_2\cdots u_ma,u_1^2,u_2^2,\dots,u_{m-1}^2,\\
v_1,v_2,\dots,v_m,
w_{1,j},w_{2,j},\dots,w_{m,j},\\
s_{1,j},s_{2,j},\dots,s_{m,j} \mid j=1,\dots,\ell)\end{array}}
for
$$u_1=t_1t_2, u_2=t_3t_4,\dots, u_m=t_{2m-1}t_{2m},$$
$$v_1=t_1^2+t_2^2,v_2=t_3^2+t_4^2,\dots, v_m=t_{2m-1}^2+t_{2m}^2,$$
$$w_{1,j}=z_{1,j}+z_{2,j}, w_{2,j}=z_{3,j}+z_{4,j},\dots,w_{m,j}=z_{2m-1,j}+z_{2m,j},$$
$$s_{1,j}=(z_{1,j}-z_{2,j})(t_1^2-t_2^2),\dots, s_{m,j}=(z_{2m-1,j}-z_{2m,j})(t_{2m-1}^2-t_{2m}^2). $$
In fact, clearly we have $\supseteq$ in \rref{efff}. On the other hand, we see that the field $E$ is an
extension of the right hand side of \rref{efff} of degree $\leq 2^{m}$. (This is because we may use quadratic
equations to solve for $t_{2i-1}^2, t_{2i}^2$ using $u_i$, $v_i$, $i=1,\dots m$, and then the $z_{i,j}$'s are
recovered from linear equations.) Thus, equality in \rref{efff} follows from Galois theory.

Setting
$$t_i'=u_i$$
$$z_{i,j}'=w_{i,j}$$
$$z_{i,j+n}'=s_{i,j}$$
$$z_{i,2n+1}'=v_i,$$
$i=1,\dots,\ell$
reduces the statement to the induction hypothesis with $k$ replaced by $m$.

\end{proof}

\vspace{3mm}
\noindent
{\bf Comment:} The above proof came from the following idea: The $D_k$ lattice (which is 
the root lattice of the type $D_k$ Lie algebra, but that fact is of no consequence of us) consists of 
$k$-tuples of integers with an even sum. 
Note that $V=Spec((t_1\cdots t_k)^{-1}R)$ where $R$ is the subring of $k$-tuples of integers whose
sum is even. 
$F[t_1,\dots,t_k]$ generated by $t_1\cdots t_k,t_1^2,t_2^2,\dots, t_{k-1}^2$ can be identified with 
the torus over $F$ corresponding to the dual lattice $D_k^*$ of $D_k$. 
The variety $V_a=Spec((t_1\cdots t_k)^{-1}R_a)$
where $R_a$ is the subring of $F[a][t_1,\dots,t_k]$ generated by $t_1\cdots t_ka,t_1^2,t_2^2,\dots, t_{k-1}^2$
is a principal homogeneous space of $V$. The induction is based on the fact that taking the quotient of
$V_a$ under the action of the abelian group $\Gamma_k$ 
generated by $\sigma_{0,s}$, $1\leq s\leq \lfloor k/2\rfloor$, with the generators acting by permutation of
coordinates (and trivially on $Spec(F[a])$), is birationally
equivalent to a product of an affine space over $F$ with a variety of the same kind with $k$ replaced by $\lfloor k/2
\rfloor$. While we are not primarily interested in 
the case when $a\notin F$, considering this case is forced by the induction in the case of
numbers $k$ other than powers of $2$.

\vspace{5mm}

{\em Proof of Theorem \ref{tbop}:} Let $K$ be a function field over $\C$ and let $R\subset K$ be a DVR with field of fractions $K$. (Then we
automatically have $\C\subseteq R$, since the valuation of any element of $\C^\times$ is infinitely divisible, and hence is $0$ in $\Z$.)
For any $\ell\in \N$, we have a {\em residue homomorphism} 
\beg{epartial}{\partial_R:H^n(K,\Z/\ell)\r H^{n-1}(k_R,\Z/\ell)
}
where $k_R$ is the residue field of $R$.
Colliot-Th\'el\`ene and Ojanguren \cite{cothe} define
{\em unramified cohomology} by 
$$H^n_{nr}(K,\Z/\ell)=\bigcap_R Ker(\partial_R).$$ 
They prove that unramified cohomology of rational fields vanishes in degrees $\geq 1$.

For a finite extension $L\supseteq K$ and a DVR $S$ with fraction field $L$ containing $R$ and residue 
field $k^\prime$, we have a commutative diagram (see \cite{GMS}, Section 8)
$$\diagram
H^n(K,\Z/\ell)\dto_{res}\rto^{\partial_R} & H^{n-1}(k,\Z/\ell)\dto^{e\cdot res}\\
H^n(L,\Z/\ell)\rto^{\partial_S} & H^{n-1}(k^\prime,\Z/\ell)
\enddiagram
$$
where $e$ is the ramification index and $res$ is the restriction
map. This makes unramified cohomology functorial with respect to restriction:
\beg{eun}{res:H^n_{nr}(K,\Z/\ell)\r H^n_{nr}(L,\Z/\ell).
}
Now consider the norm (corestriction) 
$$N:H^n(L,\Z/\ell)\r H^n(K,\Z/\ell).$$
We have 
$$N\circ res=[L:K],$$
so we have proved the following 

\begin{lemma}
If $H^n_{nr}(L,\Z/\ell)=0$, then $[L:K]$ annihilates $H^n_{nr}(K,\Z/\ell)$. 
\end{lemma}

Now consider $K=E^G$ where $G$ is a finite group and $E$ is a field of rational functions where $G$ acts
on the variables by permutations. If a group is annihilated by all primes, it is $0$. Thus, we have proved

\begin{lemma}\label{l2}
If $H^n_{nr}(E^{G_p},\Z/\ell)=0$ for all $p$-Sylow subgroups of $G$, then $H^n_{nr}(E^G,\Z/\ell)=0$.
\end{lemma}

In the case when $G=A_n$, the assumption of Lemma \ref{l2} for $p=2$ is verified by Theorem \ref{ts}.
The $p$-Sylow subgroup of $A_n$ for $p$ odd is a $p$-Sylow subgroup $G_{n,p}$ of $\Sigma_n$. In this case,
it is well known that $\C(x_1,\dots,x_n)^{G_{n,p}}$ is rational. (For a proof of this fact,
see \cite{KWZ}. Briefly, taking the Fischer \cite{fischer}
generators of the fixed field of $\C(x_1,\dots,x_n)$
under the abelian subgroup generated by $\sigma_{0,s}$, the permutations $\sigma_{i,s}$ for $i\geq 1$ act
on them also by permutation, while generating $G_{m,p}$ for some $m<n$. This gives a proof by induction.)
This verifies the assumption of Lemma \ref{l2} for odd primes $p$, and hence concludes the proof of Theorem
\ref{tbop}.
\qed

\vspace{5mm}

\section{The case $p>2$.} \label{sp}

\vspace{3mm}
In this section, we shall prove Theorem \ref{ts} for $p>2$ and Theorem \ref{tpos}. We include a proof of 
Lemma \ref{tw}, to make our proof of Theorem \ref{ts} self-contained.

\vspace{3mm}
\noindent
{\em Proof of Lemma \ref{tw}:}
Denote the primitive $p$th root of unity in $F$ by $\zeta_p$.
We will work in the field
$$L=F(x_1,\dots,x_p,w^{\frac{1}{p}}).$$
We have $\Z/p\times \Z/p=\Z/p\{g_1,g_2\}$ acting on $L$ where $g_1$
acts by 
\rref{1tw}
(and trivially on $w^{\frac{1}{p}}$) and $g_2$ acts trivially on $x_i$,
$i=1,\dots,p$ and
\beg{2tw}{g_2(w^{\frac{1}{p}})=(\zeta_p)^{-1}w^{\frac{1}{p}}.}
We have
$$K=L^{\Z/p\{g_1,g_2\}}.$$
Let
$$\bar{x_i}=\frac{x_i}{w^{\frac{1}{p}}}.$$
Then
$$L=F(\bar{x_1},\dots,\bar{x_p},w^{\frac{1}{p}})$$
and $g_1$ acts by
$$\bar{x_1}\mapsto\bar{x_2}\mapsto\dots\bar{x_{p}}\mapsto\bar{x_1},$$
while
$$\bar{x_1}\cdot\dots\cdot\bar{x_p}=1.$$
We will use the method of Chu and Kang \cite{ckang}
to describe $L^{\Z/p\{g_1\}}$.
Let
$$u_i=1+\zeta_p^i \bar{x_1}+\zeta_p^{2i}\bar{x_i}\bar{x_2}+\dots+\zeta_p^{i(p-1)}\bar{x_1}\cdot\dots\cdot\bar{x}_{p-1},$$
$i=0,\dots,p-1$.
Then
$$g_1(u_i)=\zeta_p^{-i}\frac{u_i}{\bar{x_1}}.$$
Therefore,
$$L^{\Z/p\{g_1\}}=F(w^{\frac{1}{p}})(u_0^pu_1^{-p},u_0^{-1}u_1^{2}u_2^{-1},u_1^{-1}u_2^2u_3^{-1},\dots,
u_{p-3}^{-1}u_{p-2}^{2}u_{p-1}^{-1}).$$
Let
$$v_i=u_i^{-1}u_{i+1}^2u_{i+2}^{-1},\; i=0,\dots,p-2.$$
Note that 
$$v_{p-2}^1v_{p-3}^2\cdot\dots\cdot v_0^{p-1}=u_0^{-p}u_1^p.$$
Therefore,
$$L^{\Z/p\{g_1\}}=F(w^{\frac{1}{p}})(v_0,\dots,v_{p-2})$$
where $g_2$ acts by \rref{2tw}
and
$$v_0\mapsto v_1\mapsto\dots\mapsto v_{p-2}\mapsto\frac{1}{v_0v_1\cdot\dots\cdot v_{p-2}}\mapsto v_0.$$
Let 
$$z=1+v_0+ v_0v_1+\dots+v_0\cdot\dots\cdot v_{p-2}.$$
Then
$$L^{\Z/p\{g_1\}}=F(w^{\frac{1}{p}})(\frac{1}{z},\frac{v_0}{z},\frac{v_0v_1}{z},\dots,\frac{v_0\cdot\dots\cdot v_{p-3}}{z})$$
while $g_2$ acts by \rref{2tw} and 
$$\frac{1}{z}\mapsto \frac{v_0}{z},\frac{v_0v_1}{z}\mapsto\dots$$
$$\mapsto\frac{v_0\cdot\dots\cdot v_{p-3}}{z}\mapsto 1-(\frac{1}{z}+\frac{v_0}{z}+\frac{v_0v_1}{z}+\dots+\frac{v_0\cdot\dots\cdot v_{p-3}}{z})\mapsto \frac{1}{z}.$$
Since $\Z/p\{g_2\}$ acts faithfully on $F(w^{\frac{1}{p}})$, by Theorem 1 of Hajja and Kang \cite{hajjak}, there exist
$\bar{z_1},\dots ,\bar{z}_{p-1}\in L^{\Z/p\{g_1\}}$
such that
$$ L^{\Z/p\{g_1\}}=F(w^{\frac{1}{p}})(\bar{z_1},\bar{z_2},\dots,\bar{z}_{p-1})$$
and
$$g_2(\bar{z_i})=\bar{z_i}.$$
Therefore,
$$K=F(w)(\bar{z_1},\bar{z_2},\dots,\bar{z}_{p-1}).$$

\qed

\vspace{3mm}
Using Lemma \ref{tw}, we now prove the following

\vspace{3mm}
\begin{lemma}\label{lwqm}
Let $F$ be a field of characteristic $0$ and let
$\zeta_p\in F,\; a^p\in F$.
Let
$E=F(t_1\cdot\dots\cdot t_ka,t_1^p,t_2^p,\dots,t_{k-1}^p)$
where $G_{k,p}$ acts on $t_1,\dots,t_k$ by permutation.
Then
$$E^{G_{k,p}}$$
is rational over $F$.
\end{lemma}
\begin{proof}
Induction on $k$.
If $k<p$, $E^{G_{k,p}}=E$. If
$k=mp+i$
 where $m>0$, $i=1,\dots,p-1$,
 $G_{k,p}=G_{mp,p}$.
Let
$$F'=F(t_{mp+1}^p,\dots,t_{mp+i}^p),$$
$a'=t_{mp+1}\cdot\dots\cdot t_{mp+i}a, t_{i}'=t_i$ for $i=1,\dots,mp$.
This reduces us from $k$ to $mp$. Thus assume $k=mp$.
Let
$\Gamma_k\triangleleft G_{k,p}$ be the subgroup generated
by the $p$-cycles $\sigma_{0,s}$.
Then 
$$G_{k,p}/\Gamma_k\cong G_{m,p}.$$
We have
$$E^{G_{k,p}}=(E^{\Gamma_k})^{G_{k,p}/\Gamma_k}.$$
For $i=1,\dots,m$, apply Lemma \ref{tw} to
$$x_1=t_{(i-1)p+1}^p,\dots,x_p=t_{ip}^p.$$
Let
$$u_i=t_{(i-1)p+1}\cdot\dots\cdot t_{ip}.$$
By Lemma \ref{tw},
$$F(t_{(i-1)p+1}^p,\dots,t_{ip}^p)^{\Z/p\{\sigma_{0,i}\}}=F(u_i^p,z_{i,1} ,\dots,z_{i,p-1}).$$
Therefore,
$$E^{\Gamma_k}=F(u_1u_2\cdot\dots\cdot u_ma,$$
$$u_1^p,u_2^p,\dots,u_{m-1}^p,$$
$$z_{1,1},\dots,z_{1,p-1},$$
$$\vdots$$
$$z_{m,1},\dots,z_{m,p-1}).$$
Now $G_{m,p}$ acts on the $z_{i,j}$'s by permutation of the $i$ index and faithfully
on
$$K=F(u_1u_2\cdot\dots\cdot u_ma,u_1^p,\dots,u_{m-1}^p).$$
By Theorem 1 of \cite{hajjak},
$$E^{\Gamma_k}=K(z_{1,1}',\dots,z_{1,p-1}',$$
$$ \vdots $$
$$z_{m,1}',\dots,z_{m,p-1}')$$
where $G_{m,p}$ fixes the generators $z_{i,j}'$.
Thus, our statement follows from the
induction hypothesis with $F$ replaced by
$$F(z_{1,1}',\dots, z_{1,p-1}',$$
$$ \vdots$$
$$ z_{m,1}',\dots,z_{m,p-1}').$$
\end{proof}

\vspace{3mm} 
\noindent
{\em Proof of Theorem \ref{ts} for $p>2$:}
Let $n=mp+i$ where $m>0,\; i=0,\dots,p-1$. Let 
\beg{ettlt}{t_{j}=x_{(j-1)p+1}+\zeta_{p}^{-1}x_{(j-1)p+2}+\dots +\zeta_{p}^{1-p}x_{jp},\; j=1,\dots m.}
Then the action of $H_{n,p}$ on $F(x_1,\dots, x_n)$ restricts to an action on the subfield
$$L=F(t_1,t_2,\dots,t_m)$$
where $\sigma_{0,k}$ acts by $t_k\mapsto \zeta_p t_k$, and
$t_j\mapsto t_j$ for $j\neq k,\; k=1,\dots,m$. The generators
$\sigma_{\ell,s}$ 
for $\ell>0$ act by 
permutation on the $t_k$'s.

Additionally, 
\beg{elplin}{F(x_1,\dots,x_{mp+i})=L(x_j\mid j=1,\dots, n, \; p\nmid j)}
and the action of $H_{n,p}$ on the generators on the right hand side of \rref{elplin} is
affine over $L$.
By Theorem 1 of Hajja and Kang \cite{hajjak}, $F(x_1,\dots,x_n)^{H_{n,p}}$ is rational over $L^{H_{n,p}}$. 
Thus, we may restrict attention to the action of $H_{n,p}$ on $L$.  In particular, without loss of generality, 
$i=0$ and $n=mp$.
Now if we denote by $J_n\lhd H_{n,p}$ 
the subgroup generated by
$\sigma_{0,s}\sigma_{0,s'}^{-1}$ then 
$$F(t_1,t_2,\dots,t_m)^{J_n}=F(t_1\cdot\dots\cdot t_m,t_1^p,\dots,t_{m-1}^p).$$
We have $H_{n,p}/J_n\cong G_{m,p}$, so our statement follows from Lemma \ref{lwqm} with $a=1$.
\qed

\vspace{3mm}
\noindent
{\em Proof of Theorem \ref{tpos}:} Similarly as in our proof of Theorem \ref{ts} for $p>2$,
by Theorem 1 of Hajja and Kang \cite{hajjak}, we may again instead consider
the action of $K_{n,p}$ on the subfield $L=\C(t_1,\dots,t_n)$ where the $t_j$'s are
defined by \rref{ettlt}. Then $\Sigma_n$ acts by permutation, and
$\sigma_{0,i}$ acts by $t_i\mapsto \zeta_p t_i$, $t_j\mapsto t_j$ for $j\neq i$. Thus, we must prove
that
\beg{ecsigma}{\C(t_1^p,t_2^p,\dots,t_{n-1}^p, t_1\cdot\dots\cdot t_n)^{\Sigma_n}
}
is rational. However, \rref{ecsigma} is the field of rational functions on the generators
$$\sigma_n(t_1,\dots,t_n),$$
$$\sigma_i(t_1^p,\dots,t_n^p),\; i=1,\dots,n-1$$
where $\sigma_i$ are the elementary symmetric polynomials.
\qed

\end{document}